\documentclass[10pt,reqno]{amsart}

\usepackage[latin2]{inputenc}
\usepackage{amsmath}
\usepackage{graphicx}
\usepackage{amssymb}
\usepackage{esint}
\usepackage{color}
\usepackage{amsthm}
\usepackage{epsfig}
\usepackage{enumitem}
\usepackage{mathtools}
\usepackage{esvect}
\usepackage{tikz}
\usetikzlibrary{arrows}
\usepackage[english]{babel}

\newtheorem{theorem}{Theorem}
\newtheorem{proposition}[theorem]{Proposition}
\newtheorem{lemma}[theorem]{Lemma}

\newtheorem{remark}[theorem]{Remark}
\newtheorem*{theorem*}{Theorem}
\theoremstyle{definition}

\def\XXint#1#2#3{{\setbox0=\hbox{$#1{#2#3}{\int}$ }
\vcenter{\hbox{$#2#3$ }}\kern-.6\wd0}}

\definecolor{Yellow}{rgb}{0.95,0.9,0.0} 
\definecolor{Red}{rgb}{0.8,0.1,0.1}
\definecolor{Green}{rgb}{0.1,0.65,0.2}
\definecolor{Blue}{rgb}{0.1,0.1,0.8}
\definecolor{Purple}{rgb}{0.7,0.1,0.7}
\definecolor{Grey}{rgb}{0.6,0.6,0.6}

\newcommand{\SBV}{\operatorname{SBV}}

\newcommand{\dist}{\operatorname{dist}}

\newcommand{\Id}{\operatorname{Id}}

\allowdisplaybreaks[2]

\renewcommand{\vec}[1]{{\operatorname{#1}}}

\usepackage[linktocpage=true,colorlinks=true,linkcolor=Blue,citecolor=Green]{hyperref}

\begin{document}

\title[A monotonicity formula for the Mumford-Shah functional]{A monotonicity formula for minimizers of the Mumford-Shah functional in 2d and a sharp lower bound on the energy density}

\author{Julian Fischer}
\address{Institute of Science and Technology Austria (IST Austria), Am~Campus~1, 
3400 Klosterneuburg, Austria}
\email{julian.fischer@ist.ac.at}

\begin{abstract}
We establish a new monotonicity formula for minimizers of the Mumford-Shah functional in planar domains. Our formula follows the spirit of Bucur--Luckhaus, but works with the David-L\'eger entropy instead of the energy. Interestingly, this allows for a sharp truncation threshold. In particular, our monotonicity formula is able to discriminate between points at a $C^1$ interface and any other type of singularity in terms of a finite gap in the entropy. As a corollary, we prove an optimal lower bound on the energy density around any nonsmooth point for minimizers of the Mumford-Shah functional.
\end{abstract}

\maketitle

\section{Introduction}

In this paper, we establish a new monotonicity formula for minimizers of the Mumford-Shah functional in two dimensions
\begin{align}
\label{MS}
E_{MS}[u] := \int_{\Omega \setminus J_u} |\nabla u|^2 \,dx + \mathcal{H}^1(J_u)
\end{align}
(where $\Omega\subset \mathbb{R}^2$, $u\in \SBV(\Omega)$, and where $J_u$ denotes the jump set of $u$).
The Mumford-Shah functional was originally proposed by Mumford and Shah \cite{MumfordShah} as a model for image segmentation. It serves also as an important prototypical functional for variational models in fracture mechanics, with the Dirichlet energy $\smash{\int_{\Omega \setminus J_u}} |\nabla u|^2 \,dx$ being a prototype for the stored elastic energy and with the interface energy $\mathcal{H}^1(J_u)$ accounting for the energy of the fractures; for a more detailed discussion of variational models for fracture, we refer to \cite{FrancfortMarigo,DalMasoFrancfortToader,ChambolleContiIurlano}.

The variational problem for the Mumford-Shah functional is one of the most prominent examples of a free discontinuity problem, an important class of problems in the calculus of variations in which the discontinuity set of a minimizer is itself subject to optimization.
Existence of minimizers for the Mumford-Shah functional (subject to suitable boundary conditions and possibly with an additional lower-order term added to \eqref{MS}) has been established by Ambrosio \cite{Ambrosio}, De~Giorgi, Carriero, and Leaci \cite{DeGiorgiCarrieroLeaci}, and Dal~Maso, Morel, and Solimini \cite{DalMasoMorelSolimini}.
For an introduction to free discontinuity problems, we refer to \cite{AmbrosioFuscoPallara}.

Our first main result states that any minimizer $u\in \SBV(\Omega)$ of the Mumford-Shah functional \eqref{MS} is subject to a monotonicity formula of the form
\begin{align}
\label{MonotonicityFormula}
\frac{d}{dr} \min\Big\{F(r,x_0),\frac{3}{2}\Big\}
\geq \frac{1}{r} D(r,x_0) \geq 0
\end{align}
where $F(r,x_0)$ denotes the \emph{David-L\'eger entropy}\footnote{In order to distinguish the functional \eqref{Entropy} from the rescaled Mumford-Shah energy on the ball $B_r(x_0)$ as defined in \eqref{EnergyDensity}, we shall refer to it as the \emph{David-L\'eger entropy}, having been introduced by David and L\'eger \cite{DavidLeger}.}
\begin{align}
\label{Entropy}
F(r,x_0):=\frac{1}{r}\bigg[\int_{B_r(x_0) \setminus J_u} |\nabla u|^2 \,dx + \frac{1}{2}\mathcal{H}^1\big(J_u\cap B_r(x_0)\big) \bigg]
\end{align}
and where $D(r,x_0)$ denotes a nonnegative functional detailed below.
To be precise, the inequality \eqref{MonotonicityFormula} holds for all $x_0\in \Omega$ and all $r>0$ such that $B_r(x_0)\subset \Omega$, with the derivative $\frac{d}{dr}$ and the inequality understood in the distributional sense.

The example of the \emph{crack-tip} (given in radial coordinates; see Figure~\ref{FigureCrackTip} for an illustration)
\begin{align}
\label{CrackTip}
u(r,\varphi):=\sqrt{\tfrac{2}{\pi}} \, r^{1/2} \cos \tfrac{\varphi}{2},
\qquad 0<\varphi<2\pi,
\end{align}
-- which is a global minimizer of the Mumford-Shah functional as proven by Bonnet and David \cite{BonnetDavid} --
demonstrates that the truncation threshold $\frac{3}{2}$ for the entropy $F(r,x_0)$ is chosen in a sharp way, i.\,e.\ no monotonicity formula analogous to \eqref{MonotonicityFormula} can hold for a quantity of the form $\min\{F(r,x_0),a\}$ when $a$ is chosen as $a>\frac{3}{2}$.
We refer to Remark~\ref{RemarkSharpThreshold} below for a more detailed discussion of this elementary observation.

As a consequence of our estimates and our monotonicity formula, we shall prove that a minimizer of the Mumford-Shah energy \eqref{MS} must satisfy the energy density lower bound
\begin{align}
\label{DensityLowerBoundIntroduction}
E(r,x_0):=\frac{1}{r}\bigg[\int_{B_r(x_0) \setminus J_u} |\nabla u|^2 \,dx + \mathcal{H}^1(J_u\cap B_r(x_0)) \bigg]\geq 2
\end{align}
around any singular point $x_0$ and for all $r>0$ such that $B_r(x_0)\subset \Omega$. In other words, as soon as $E(r,x_0)<2$ holds for some $x_0\in \Omega$ and some $r\in (0,\dist(x_0,\partial\Omega))$ it follows that $u$ is analytic in a neighborhood of $x_0$.
As the explicit examples of the crack-tip \eqref{CrackTip} and the planar interface
\begin{align*}
u(x):=
\begin{cases}
\alpha&\text{if } (x-x_0)\cdot \vec{n}>0,
\\
\beta&\text{else},
\end{cases}
\end{align*}
(with an arbitrary normal vector $\vec{n}$ and any $\alpha,\beta\in \mathbb{R}$, $\alpha\neq \beta$; note that it is a minimizer on a bounded domain $\Omega$ provided that $|\alpha-\beta|\geq C(\Omega)$, see e.\,g.\ \cite{AlbertiDalMasoCalibration}) show, the lower bound \eqref{DensityLowerBoundIntroduction} is sharp.

We are only aware of a single unconditional\footnote{as opposed to the monotonicity of the Dirichlet energy by Bonnet \cite{Bonnet} which holds only under the additional assumption that the discontinuity set is connected} monotonicity formula for minimizers of the Mumford-Shah functional that was established prior to our work, namely the monotonicity formula by Bucur and Luckhaus \cite{BucurLuckhaus}; it states that there exists a universal constant $c=c(d)>0$ such that the truncated rescaled energy $\min\{E(r,x_0),c\}$ is nondecreasing in $r$ for any minimizer.
Here, the rescaled energy is defined as
\begin{align}
\label{EnergyDensity}
E(r,x_0):=\frac{1}{r} \bigg[\int_{B_r(x_0)\setminus J_u} |\nabla u|^2 \,dx + \mathcal{H}^1(J_u\cap B_r)\bigg].
\end{align}
As the truncation threshold $c>0$ is rather small, their result only provides a far-from-optimal lower bound on the energy density \eqref{EnergyDensity} near singular points. The stronger energy density lower bound $E(r,x_0)\geq 1$ has been obtained by De~Lellis and Focardi \cite{DeLellisFocardi}; however, their estimate was still suboptimal by a factor of $2$.

Before turning to the statement of our main theorems, let us briefly comment on the regularity properties of minimizers of the Mumford-Shah functional.
Partial regularity of minimizers has been established by Ambrosio, Fusco, and Pallara \cite{AmbrosioPallaraPartialRegularity,AmbrosioFuscoPallaraPartialRegularity}.
There are several intriguing and widely open conjectures concerning the regularity of minimizers; most famously, in the planar case $d=2$ it has been conjectured by Mumford and Shah \cite{MumfordShah} that minimizers should be smooth outside of a locally finite collection of $C^1$ arcs, a statement known as the Mumford-Shah conjecture. In \cite{DeGiorgiConjectures}, De~Giorgi formulated several further conjectures regarding the singular set of minimizers. A result relating higher integrability of the gradient $\nabla u$ to the dimension of the singular set has been shown by Ambrosi, Fusco, and Hutchinson \cite{AmbrosioFuscoHutchinson}; higher integrability of the gradient $\nabla u$ in turn has been established by De~Lellis and Focardi \cite{DeLellisFocardiHigherIntegrability} for $d=2$ and by De~Philippis and Figalli \cite{DePhilippisFigalli} in $d\geq 3$. Recently, results on higher regularity near endpoints of curves in the discontinuity set $J_u$ have been obtained, see \cite{AnderssonMikayelyanEndpoint,DeLellisFocardiGhinassiEndpoint}.


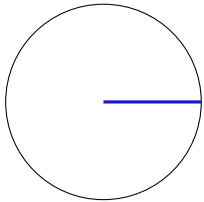
\begin{figure}
\begin{tikzpicture}[scale=0.65]
\draw (0,0) circle (2.0);
\draw[color=Blue,very thick] (0,0) -- (2,0);
\end{tikzpicture}
\caption{
\label{FigureCrackTip}
The discontinuity set (blue) of the crack-tip \eqref{CrackTip} in the ball $B_r(x_0)$, $x_0=0$. Note that the density of the Dirichlet energy of the crack-tip is given by $|\nabla u(\rho,\varphi)|^2=\rho^{-1}$.
}
\end{figure}

{\bf Notation.}
Throughout the paper, we use standard notation for functions of bounded variation. In particular, by $\SBV(\Omega)$ we denote the space of special functions of bounded variation, defined as those functions of bounded variation for which the Cantor part of the distributional derivative vanishes. For $u\in \SBV(\Omega)$, we denote by $J_u$ its jump set (see \cite{AmbrosioFuscoPallara} for the definition) and by $\nabla u$ the absolutely continuous part of its distributional derivative. By the fine properties of $\SBV$ functions, the jump set $J_u$ of a function $u\in \SBV$ coincides with its approximate discontinuity set $S_u$ up to a set of negligible $\mathcal{H}^{d-1}=\mathcal{H}^1$ measure.

By $B_r(x_0)$, we denote the ball of radius $r$ with center $x_0$. For the outward unit normal on $\partial B_r(x_0)$ we will use the notation $\nu(x):=\smash{\tfrac{x-x_0}{|x-x_0|}}$; by $\nu^\perp$ we denote the tangential vector to $\partial B_r(x_0)$, with the convention $e_1^\perp=e_2$ for the two vectors $e_1,e_2$ of the standard basis.
The tangential and normal components of the (absolutely continuous part of the) weak derivative $\nabla u$ on $\partial B_r(x_0)$ will be abbreviated as $\partial_\nu u := \nu(x) \cdot \nabla u$ and $\partial_\tau u := \nu^\perp(x) \cdot \nabla u$; note that $|\partial_\nu u|^2 + |\partial_\tau u|^2 = |\nabla u|^2$. For the tangent vector field $\vec{t}_{J_u}$ of the jump set, we use the sign convention that $\nu\cdot \vec{t}_{J_u}\geq 0$.

Furthermore, we frequently use indicator functions of the form $\smash{\chi_{F(r,x_0)<\tfrac{3}{2}}}$, defined to be $1$ if $F(r,x_0)<\tfrac{3}{2}$ and to be $0$ otherwise.

\section{Main Results}

As our first main result, we prove the following monotonicity formula for minimizers of the Mumford-Shah functional \eqref{MS}.
\begin{theorem}[Monotonicity formula]
\label{TheoremMonotonicityFormula}
Let $\Omega\subset \mathbb{R}^2$ be a planar domain and let $u\in \SBV(\Omega)$ be a minimizer of the Mumford-Shah functional \eqref{MS} subject to any boundary conditions. Fix $x_0\in \Omega$ and define
\begin{align*}
F(r,x_0):=\frac{1}{r}\bigg[\int_{B_r(x_0) \setminus J_u} |\nabla u|^2 \,dx + \frac{1}{2}\mathcal{H}^1(J_u\cap B_r(x_0)) \bigg].
\end{align*}
Then the monotonicity formula
\begin{align*}
\frac{d}{dr} \min\Big\{F(r,x_0),\frac{3}{2}\Big\}
\geq \frac{1}{r}D(r,x_0) \geq 0
\end{align*}
holds in the distributional sense for $0<r<\dist(x_0,\partial\Omega)$, where $D(r,x_0)$ is estimated from below by
\begin{align*}
D(r,x_0)\geq
\begin{cases}
0&\text{if }F(r,x_0)\geq \frac{3}{2},
\\
\frac{1}{C}\int_{\partial B_r(x_0)\setminus J_u} |\nabla u|^2 \,d\mathcal{H}^1 &\text{if }\int_{\partial B_r(x_0)\setminus J_u} |\nabla u|^2 \,d\mathcal{H}^1 \geq 2\text{ and }F(r,x_0)<\frac{3}{2},
\\
\frac{1}{C}\int_{\partial B_r(x_0)\setminus J_u} |\nabla u|^2 \,d\mathcal{H}^1 &\text{if }\#\{J_u\cap \partial B_r(x_0)\}\in \{0,2\}\text{ and }F(r,x_0)<\frac{3}{2},
\\
\frac{3}{2}-F(r,x_0) &\text{if }\#\{J_u\cap \partial B_r(x_0)\}\notin \{0,2\}\text{ and }F(r,x_0)<\frac{3}{2}.
\end{cases}
\end{align*}
Here, for a set $A$ we denote by $\#A$ the number of its elements.
\end{theorem}
In particular, our monotonicity formula is able to distinguish between smooth points, points at which $u$ has a discontinuity across a $C^1$ interface, and any other type of singularity in terms of a finite gap in the entropy: It ensures that at each point $x_0\in \Omega$ one of the three alternatives
\begin{itemize}
\item $\lim_{r\rightarrow 0} F(r,x_0)=0$, implying that $u$ is analytic in a neighborhood of $x_0$,
\item $F(r,x_0)\geq 1$ for all $r\in (0,\dist(x_0,\partial\Omega))$ and $\lim_{r\rightarrow 0} F(r,x_0)=1$, implying that near $x_0$ the discontinuity set $J_u$ is given by a $C^{1,1}$ interface and that $\nabla u$ is bounded near $x_0$,
\item $F(r,x_0)\geq \frac{3}{2}$ for all $r\in (0,\dist(x_0,\partial\Omega))$, corresponding to any other type of singularity,
\end{itemize}
must hold\footnote{To see this, note that our monotonicity formula entails that either $F(r,x_0)\geq \frac{3}{2}$ holds for all $r>0$ or that $\smash{\lim_{r\rightarrow 0}\frac{1}{r}\int_{B_r(x_0)}|\nabla u|^2 \,dx} = 0$. In the latter case, it follows from \cite[Proposition~5.8]{AmbrosioFuscoHutchinson} that one of out three options holds near $x_0$: $u$ is analytic near $x_0$, $J_u$ is locally a $C^{1,1}$ curve passing through $x_0$ with a uniform bound on $\nabla u$ in a neighborhood, or the blowup limits of $J_u$ around $x_0$ converge to a ``propeller'', which consists of three lines meeting at $x_0$. The second alternative clearly implies $\lim_{r\rightarrow 0} F(r,x_0)=1$, while the last alternative entails $\liminf_{r\rightarrow 0} F(r,x_0)\geq \frac{3}{2}$.}. Note that a criterion based on the energy density $E(r,x_0)$ as given by \eqref{EnergyDensity} could not distinguish between a crack-tip at $x_0$ and $x_0$ being part of a $C^{1,1}$ interface, as in both cases it holds that $\lim_{r\rightarrow 0} E(r,x_0)=2$.

As an inspection of our proofs reveals, our monotonicity formula also holds for reduced global minimizers of the Mumford-Shah functional on $\Omega=\mathbb{R}^2$ (which arise as blowup limits of minimizers around singular points; e.\,g.\ \cite{DavidLeger} for an explanation of the concept and its role in the regularity theory of minimizers). The same in fact applies to our density lower bound stated in the next theorem.

As a corollary to our monotonicity formula and to the key new ingredient of its proof, we obtain a sharp lower bound on the energy density $E(r,x_0)$ around singularities of minimizers. Note that the previously available lower bound established by De~Lellis and Focardi \cite{DeLellisFocardi} is suboptimal by a factor of $2$, i.\,e.\ it only provides the estimate $E(r,x_0)\geq 1$ around singular points instead of our sharp estimate $E(r,x_0)\geq 2$.
\begin{theorem}[Sharp lower bound for the energy density at singularities]
\label{TheoremDensityLowerBound}
Let $\Omega\subset \mathbb{R}^2$ be a planar domain and let $u$ be a minimizer of the Mumford-Shah energy \eqref{MS} subject to any boundary conditions on $\partial\Omega$. Then for any point $x_0 \in \Omega$, one of the following two alternatives holds:
\begin{itemize}
\item[i)] $u$ is analytic in a neighborhood of $x_0$, i.\,e.\ $u$ has no singularity near $x_0$.
\item[ii)] $u$ has a singularity at $x_0$ and the density lower bound
\begin{align}
\label{DensityLowerBound}
E(r,x_0):=
\frac{1}{r}\bigg[\int_{B_r(x_0) \setminus J_u} |\nabla u|^2 \,dx + \mathcal{H}^1(J_u\cap B_r(x_0)) \bigg] \geq 2
\end{align}
holds for all $0<r<\dist(x_0,\partial \Omega)$.
\end{itemize}
\end{theorem}

We finally mention the example showing that the truncation threshold $\tfrac{3}{2}$ in our monotonicity formula \eqref{MonotonicityFormula} is sharp.
\begin{remark}
\label{RemarkSharpThreshold}
No monotonicity formula analogous to \eqref{MonotonicityFormula} can hold for a quantity of the form $\min\{F(r,x_0),a\}$ when $a$ is chosen as $a>\frac{3}{2}$.
To see this, consider the entropy $F(r,x_0)$ of the crack-tip $u(\rho,\varphi):=\smash{\sqrt{2/\pi} \, \rho^{1/2}} \cos \tfrac{\varphi}{2}$ (with $\rho>0$, $0<\varphi<2\pi$) for different (large) radii $r$ around the point $x_0:=e_1$ on the horizontal axis.
It is easy to compute that for the crack-tip one has $F(r,e_1)=F(1,r^{-1} e_1)$ for any $r>0$. In other words, one has $\lim_{r\rightarrow \infty} F(r,e_1)=F(1,0)=\tfrac{3}{2}$. To show the failure of monotonicity of $\min\{F(r,x_0),a\}$ whenever $a$ is chosen as $a>\tfrac{3}{2}$, it therefore suffices to show $F(r,e_1)>\tfrac{3}{2}$ for some large enough $r$. Due to $F(r,e_1)=F(1,r^{-1} e_1)$, this is equivalent to showing $F(1,\delta e_1)>\frac{3}{2}$ for some small enough $\delta>0$.
The latter inequality in fact holds for all sufficiently small $\delta>0$: A Taylor expansion yields $\int_{B_1(\delta e_1)} |\nabla u|^2 \,dx\geq \int_{B_1(0)} |\nabla u|^2 \,dx - O(\delta^2)=1-O(\delta^2)$ for small $\delta$, as the crack-tip satisfies $|\nabla u|^2(\rho,\varphi)=\tfrac{1}{2\pi}\rho^{-1}$. On the other hand, a direct computation gives $\mathcal{H}^1(J_u\cap B_1(\delta e_1))=1+\delta$. Overall, one has $F(1,\delta e_1)\geq\tfrac{3}{2}+\tfrac{\delta}{2}-O(\delta^2)$, showing that $F(1,\delta e_1)>\tfrac{3}{2}$ for $\delta>0$ small enough.
\end{remark}

\section{A lower bound on the Dirichlet energy on circles}
\label{SectionNewFormula}

The following proposition provides a lower bound on the Dirichlet energy of minimizers $u$ on circles $\partial B_r(x_0)$ in terms of the tangent vectors of the singular set $J_u$ at its intersections with $\partial B_r(x_0)$. In particular, whenever $J_u$ intersects $\partial B_r(x_0)$ at exactly one point, we obtain the lower bound $\int_{\partial B_r(x_0)} |\nabla u|^2 \,d\mathcal{H}^1\geq 1$. This estimate is at the very heart of the proof of our monotonicity formula and our sharp density lower bound.
\begin{proposition}
\label{PropositionLowerBoundRadialDirichletEnergyCrack}
Let $\Omega\subset \mathbb{R}^2$ and let $u$ be a minimizer of the Mumford-Shah energy \eqref{MS}. Let $x_0 \in \Omega$.
Then for a.\,e.\ $0<r<\dist(x_0,\partial\Omega)$
the estimate
\begin{align*}
\int_{\partial B_r(x_0) \setminus J_u} |\nabla u|^2 \,d\mathcal{H}^1
\geq \sum_{x\in J_u \cap \partial B_r(x_0)} \vec{q} \cdot \vec{t}_{J_u}(x)
\end{align*}
holds for any unit vector $\vec{q}\in \mathbb{S}^1$, where we use the convention that $\vec{t}_{J_u}$ is oriented such that $\nu(x)\cdot \vec{t}_{J_u}(x)>0$ for $x\in J_u \cap \partial B_r(x_0)$.

In particular, for a.\,e.\ $0<r<\dist(x_0,\partial\Omega)$ with $\# \{J_u\cap \partial B_r(x_0)\}=1$, we have
\begin{align*}
\int_{\partial B_r(x_0) \setminus J_u} |\nabla u|^2 \,d\mathcal{H}^1
\geq 1.
\end{align*}
\end{proposition}
Before beginning with the proof, let us provide an illustration of its idea.
In the situation of the second estimate in Proposition~\ref{PropositionLowerBoundRadialDirichletEnergyCrack}, the jump set $J_u$ of the minimizer intersects the disk boundary $\partial B_r(x_0)$ at precisely one point. The assertion is then that the Dirichlet energy of $u$ integrated over the disk boundary $\partial B_r(x_0)$ must be at least $1$.
The idea for the proof is to consider a variation by a family of diffeomorphisms whose generating vector field vanishes outside of the disk $B_r(x_0)$ and is constant inside the disk $B_r(x_0)$, ``pushing'' the jump set $J_u$ out of $B_r(x_0)$ (at its only intersection with $\partial B_r(x_0)$). Such a variation attempts to reduce the interface energy in $\overline{B_r(x_0)}$ by reducing the interface length near $\partial B_r(x_0)$; for $u$ to be a minimizer, it turns out that the Dirichlet energy on the disk boundary $\int_{\partial B_r(x_0) \setminus J_u} |\nabla u|^2 \,d\mathcal{H}^1$ must be at least $1$ to counteract this energy reduction.
\begin{proof}[Proof of Proposition~\ref{PropositionLowerBoundRadialDirichletEnergyCrack}]
Any minimizer $u$ of the Mumford-Shah functional \eqref{MS} satisfies the equilibrium equation
\begin{align}
\label{EquilibriumEquation}
\int_{\Omega\setminus J_u} \big(2 \nabla u \otimes \nabla u - |\nabla u|^2 \Id \big) : \nabla \eta \,dx
-\int_{\Omega \cap J_u} \vec{t}_{J_u} \otimes \vec{t}_{J_u} : \nabla \eta \,d\mathcal{H}^1 = 0
\end{align}
for any compactly supported $C^1$ vector field $\eta$ (see e.\,g.\ \cite{AmbrosioFuscoPallara} for a derivation).
Given a unit vector $\vec{q}\in \mathbb{S}^1$, we set (making use of an approximation argument)
\begin{align*}
\eta:=
\begin{cases}
0&\text{for }|x|\geq r,
\\
\frac{r-|x|}{\delta} \vec{q}&\text{for }r\geq |x|\geq r-\delta,
\\
\vec{q}&\text{for }|x|\leq r-\delta.
\end{cases}
\end{align*}
Computing that
\begin{align*}
\nabla \eta
=
\begin{cases}
0&\text{for }|x|\geq r,
\\
-\frac{1}{\delta} \vec{q} \otimes \frac{x}{|x|}&\text{for }r\geq |x|\geq r-\delta,
\\
0&\text{for }|x|\leq r-\delta
\end{cases}
\end{align*}
and passing to the limit $\delta\searrow 0$, this yields\footnote{To make this argument rigorous, one makes use of the corresponding arguments from the proof of Lemma~\ref{LemmaBeginMonotonicity} below. Note that the prefactor $\smash{\frac{1}{|\vec{t}_{J_u}(x) \cdot \nu(x)|}}$ in the second term stems from the coarea formula for the rectifiable set $J_u$.} for a.\,e.\ $r\in (0,\dist(x_0,\partial\Omega))$
\begin{align*}
&-\int_{\partial B_r(x_0)\setminus J_u} \big(2 \nabla u \otimes \nabla u - |\nabla u|^2 \Id \big) : (\vec{q} \otimes \nu) \,d\mathcal{H}^1
\\&
+\sum_{x\in J_u\cap \partial B_r(x_0)} \frac{1}{|\vec{t}_{J_u}(x) \cdot \nu(x)|} \vec{t}_{J_u}(x)\otimes \vec{t}_{J_u}(x) : (\vec{q} \otimes \nu(x)) =0
\end{align*}
with the (outward) unit normal vector field $\nu$ of $\partial B_r(x_0)$.
For computational convenience and without loss of generality (due to translational and rotational invariance of the problem and the statement of the proposition), from now on assume that $x_0=0$ and $\vec{q}=e_1$. Transforming to radial coordinates, we obtain with $\tau:=\nu^\perp$, $\partial_\nu u := \nu\cdot \nabla u$, and $\partial_\tau u := \tau \cdot \nabla u$ (and thus $\nabla u=\partial_\nu u\,\nu + \partial_\tau u\,\tau$)
\begin{align*}
&r\int_{0}^{2\pi} 2 (e_1 \cdot \nu) |\partial_\nu u|^2 + 2(e_1 \cdot \tau) \partial_\tau u \, \partial_\nu u - |\nabla u|^2 e_1 \cdot \nu \,d\varphi
=
\sum_{x\in J_u\cap \partial B_r(x_0)} e_1 \cdot \vec{t}_{J_u}(x).
\end{align*}
Using $\nu=(\cos \varphi,\sin \varphi)$ and $\tau=(-\sin \varphi,\cos \varphi)$, this may be rewritten as
\begin{align*}
r\int_{0}^{2\pi} \cos(\varphi) (|\partial_\nu u|^2-|\partial_\tau u|^2) - 2 \sin(\varphi) \partial_\tau u \, \partial_\nu u \,d\varphi
=
\sum_{x\in J_u\cap \partial B_r(x_0)} e_1 \cdot \vec{t}_{J_u}(x).
\end{align*}
The elementary inequality $-2\sin(\varphi) \partial_\tau u \, \partial_\nu u \leq (1+\cos(\varphi)) |\partial_\tau u|^2+\frac{\sin^2(\varphi)}{1+\cos(\varphi)}|\partial_\nu u|^2$ and the elementary identity $\frac{\sin^2(\varphi)}{1+\cos(\varphi)}=1-\cos(\varphi)$ yield
\begin{align*}
r\int_{0}^{2\pi} |\partial_\nu u|^2+|\partial_\tau u|^2 \,d\varphi
\geq \sum_{x\in J_u\cap \partial B_r(x_0)} e_1 \cdot \vec{t}_{J_u}(x)
\end{align*}
and hence
\begin{align*}
\int_{\partial B_r(x_0)\setminus J_u} |\nabla u|^2 \,d\mathcal{H}^1
\geq
\sum_{x\in J_u\cap \partial B_r(x_0)} e_1 \cdot \vec{t}_{J_u}(x).
\end{align*}
By rotational invariance of the problem, the latter estimate holds true for any choice $\vec{q}\in \mathbb{S}^1$ in place of $e_1$, proving the first part of our proposition. In particular, in the case $\#\{J_u\cap \partial B_r(x_0)\}=1$ we may choose $\vec{q}=\vec{t}_{J_u}(x')$, establishing the second estimate in the proposition.
\end{proof}

\section{Proof of the sharp density lower bound}

With the lower bound on the Dirichlet energy on circles at hand and assuming for the moment the validity of our monotonicity formula, we establish the sharp density lower bound stated in Theorem~\ref{TheoremDensityLowerBound}.
\begin{proof}[Proof of Theorem~\ref{TheoremDensityLowerBound}]
We in fact prove the slightly stronger statement that if $x_0$ is a singular point of the minimizer $u$, we have for a.\,e.\ $0<r<\dist(x_0,\partial\Omega)$
\begin{align}
\label{LowerBoundRadialSlice}
\int_{\partial B_r \setminus J_u} |\nabla u|^2 \,d\mathcal{H}^1+\#\{J_u\cap \partial B_r(x_0)\} \geq 2.
\end{align}
Integrating the estimate \eqref{LowerBoundRadialSlice} with respect to $r$ directly entails Theorem~\ref{TheoremDensityLowerBound}, as
\begin{align*}
\mathcal{H}^1(J_u \cap B_R(x_0))
\geq \int_0^R \sum_{x\in J_u \cap \partial B_r(x_0)} \frac{1}{|\nu(x)\cdot \vec{t}_{J_u}(x)|} \,dr
\end{align*}
(see the proof of Lemma~\ref{LemmaBeginMonotonicity} below for a sketch of the argument for the latter estimate).

To prove \eqref{LowerBoundRadialSlice}, we distinguish three cases:
\begin{itemize}
\item In case $\#\{J_u\cap \partial B_r(x_0)\}\geq 2$, the estimate \eqref{LowerBoundRadialSlice} is immediate.
\item In case $\#\{J_u\cap \partial B_r(x_0)\}=1$, the estimate \eqref{LowerBoundRadialSlice} follows directly from Proposition~\ref{PropositionLowerBoundRadialDirichletEnergyCrack}.
\item It only remains to consider the case $\#\{J_u\cap \partial B_r\}=0$.
\end{itemize}

The remainder of the proof is thus concerned with the argument for \eqref{LowerBoundRadialSlice} in the third case $\#\{J_u\cap \partial B_r\}=0$.

We first observe that around a singular point $x_0$ we must have $F(r,x_0)\geq 1$ for any $r\in (0,\dist(x_0,\partial \Omega))$: If we have $F(r,x_0)<1$ for some radius $r$, the monotonicity formula from Theorem~\ref{TheoremMonotonicityFormula} implies $\liminf_{\rho \searrow 0} \frac{1}{\rho}\int_{B_\rho(x_0)} |\nabla u|^2 \,dx=0$. This in turn would entail that one of the following two alternatives holds: $u$ is either smooth in a neighborhood of $x_0$, or we have $x_0\in J_u$ with $J_u$ being of class $C^{1,1}$ in a neighborhood of $x_0$ \cite{AmbrosioFuscoPallaraPartialRegularity}. The former possibility is excluded by our assumption that $u$ has a singularity at $x_0$; the latter possibility implies $\liminf_{\rho \searrow 0} F(\rho) \geq \liminf_{\rho \searrow 0} \frac{\mathcal{H}^1(J_u\cap B_r(x_0))}{2r} = 1$, which would give $F(r,x_0)\geq 1$ for all $0<r<\dist(x_0,\partial\Omega)$ by monotonicity.

Next, we estimate the energy $E(r,x_0)$ by a competitor argument. Indeed, the standard competitor construction in Lemma~\ref{LemmaCompetitorZeroIntersection} below entails that in our case of $\#\{J_u\cap \partial B_r\}=0$ we must have
\begin{align*}
E(r,x_0) \leq \int_{\partial B_r(x_0)} |\partial_\tau u|^2 \,d\mathcal{H}^1
\end{align*}
since $u$ is a minimizer.
Furthermore, by the David--L\'eger--Maddalena--Solimini relation \eqref{DavidLegerEtAl} and $\#\{J_u\cap \partial B_r\}=0$ we get
\begin{align*}
\frac{\mathcal{H}^1(J_u\cap B_r(x_0))}{r}=\int_{\partial B_r(x_0)} |\partial_\tau u|^2-|\partial_\nu u|^2 \,d\mathcal{H}^1.
\end{align*}
The previous two formulas imply
\begin{align*}
F(r,x_0) =  E(r,x_0) - \frac{1}{2}\cdot \frac{\mathcal{H}^1(J_u\cap B_r(x_0))}{r}
\leq \frac{1}{2} \int_{\partial B_r(x_0)} |\partial_\tau u|^2+|\partial_\nu u|^2 \,d\mathcal{H}^1.
\end{align*}
Recalling that $F(r,x_0)\geq 1$, this entails \eqref{LowerBoundRadialSlice}.
\end{proof}

\begin{lemma}
\label{LemmaCompetitorZeroIntersection}
Let $u\in \SBV(\Omega)$ be such that $J_u\cap \partial B_r(x_0)=\emptyset$ and $\int_{\partial B_r(x_0)} |\partial_\tau u|^2 \,d\mathcal{H}^1<\infty$. Then there exists a function $v\in H^1(B_r(x_0))$ with $u|_{\partial B_r(x_0)}=v|_{\partial B_r(x_0)}$ satisfying
\begin{align*}
\frac{1}{r}\int_{B_r(x_0)} |\nabla v|^2 \,dx \leq \int_{\partial B_r(x_0)} |\partial_\tau u|^2 \,d\mathcal{H}^1.
\end{align*}
\end{lemma}
\begin{proof}
In order to construct such an extension $v$ of the boundary data of $u$ on $\partial B_r(x_0)$, we decompose $u|_{\partial B_r(x_0)}$ into its Fourier modes
\begin{align*}
(u|_{\partial B_r(x_0)})(r,\varphi) = \sum_{k\geq 0} a_k \cos(k\varphi)+\sum_{k\geq 1} b_k \sin(k\varphi)
\end{align*}
(with $(a_k)_k, (b_k)_k\in \ell_2$) and observe that
\begin{align*}
r\int_{\partial B_r(x_0)} |\partial_\tau u|^2 \,d\mathcal{H}^1
=\sum_{k\geq 1} \pi k^2 |a_k|^2 + \sum_{k\geq 1} \pi k^2 |b_k|^2.
\end{align*}
We then build an extension $v$ as
\begin{align*}
v(\rho,\varphi) := \sum_{k\geq 0} a_k \Big(\frac{\rho}{r}\Big)^{k} \cos(k\varphi)+\sum_{k\geq 1} b_k \Big(\frac{\rho}{r}\Big)^{k} \sin(k\varphi)
\end{align*}
and observe that
\begin{align*}
\int_{B_r(x_0)} |\nabla v|^2 \,dx
=\sum_{k\geq 1} \pi k |a_k|^2+\sum_{k\geq 1} \pi k |b_k|^2.
\end{align*}
The previous two formulas imply the estimate of the lemma.
\end{proof}

The following formula has been discovered by David and L\'eger \cite{DavidLeger} and Maddalena and Solimini \cite{MaddalenaSolimini}. 
\begin{lemma}
Let $\Omega\subset \mathbb{R}^2$ be a planar domain and let $u\in \SBV(\Omega)$ be a minimizer of the Mumford-Shah functional \eqref{MS}.
Then the David--L\'eger--Maddalena--Solimini relation
\begin{align}
\label{DavidLegerEtAl}
&\int_{\partial B_r(x_0)} |\partial_\tau u|^2 \,d\mathcal{H}^1
+\sum_{x\in J_u\cap \partial B_r(x_0)} |\nu(x) \cdot \vec{t}_{J_u}(x)|
\\&
\nonumber
= \int_{\partial B_r(x_0)} |\partial_\nu u|^2 \,d\mathcal{H}^1 + \frac{1}{r} \mathcal{H}^1(J_u\cap B_r(x_0))
\end{align}
holds for a.\,e.\ $0<r<\dist(x_0,\partial\Omega)$.
\end{lemma}
\begin{proof}
We recall that this formula may be derived from the equilibrium equation \eqref{EquilibriumEquation} by considering the family of variations 
\begin{align*}
\eta:=
\begin{cases}
0&\text{for }|x|\geq r,
\\
\frac{r-|x|}{\delta} x&\text{for }r\geq |x|\geq r-\delta,
\\
x&\text{for }|x|\leq r-\delta,
\end{cases}
\end{align*}
and letting $\delta\searrow 0$. We omit the technical details, as the formal derivation is straightforward and the technical details may be found e.\,g.\ in \cite{DeLellisFocardi}.
\end{proof}

\section{Proof of the monotonicity formula}

Our proof of the monotonicity formula combines the ideas of David-L\'eger \cite{DavidLeger} with the lower bound on the Dirichlet energy on circles $\partial B_r(x_0)$ established in Section~\ref{SectionNewFormula}.

\begin{lemma}
\label{LemmaBeginMonotonicity}
Let $\Omega\subset \mathbb{R}^2$ be a planar domain, let $x_0\in \Omega$, and let $u$ be a minimizer of the Mumford-Shah functional on $\Omega$. Then the David-L\'eger entropy \eqref{Entropy} is subject to the estimate
\begin{subequations}
\begin{align}
\label{BeginMonotonicity}
&\frac{d}{dr} \min\Big\{F(r,x_0),\frac{3}{2}\Big\}
\geq \frac{1}{r} D(r,x_0)
\end{align}
in the sense of distributions, where
\begin{align}
\label{DefinitionD1}
D(r,x_0):&=
\chi_{F(r,x_0)<\tfrac{3}{2}} \bigg(\int_{\partial B_r(x_0)\setminus J_u} |\partial_\tau u|^2 + |\partial_\nu u|^2 \,d\mathcal{H}^1
\\&~~~~~~~~~~~~~~~~~
\nonumber
+\frac{1}{2}\sum_{x\in J_u \cap \partial B_r(x_0)} \frac{1}{|\nu(x) \cdot \vec{t}_{J_u}(x)|} - F(r,x_0) \bigg).
\end{align}
Additionally, for a.\,e.\ $0<r<\dist(x_0,\partial \Omega)$ we have the representation
\begin{align}
\label{DefinitionD2}
D(r,x_0) &= \chi_{F(r,x_0)<\tfrac{3}{2}} \bigg(\int_{\partial B_r(x_0)\setminus J_u} \frac{3}{2} |\partial_\tau u|^2 + \frac{1}{2} |\partial_\nu u|^2 \,d\mathcal{H}^1
\\&~~~~~~~~~~~~~~~~~~
\nonumber
+\frac{1}{2}\sum_{x\in J_u \cap \partial B_r(x_0)} \bigg( \frac{1}{|\nu(x) \cdot \vec{t}_{J_u}(x)|} + |\nu(x) \cdot \vec{t}_{J_u}(x)|\bigg) - E(r,x_0) \bigg).
\end{align}
\end{subequations}
\end{lemma}
\begin{proof}
The statement and the proof of the lemma differ from \cite[Proposition~2.5 and formulas (2.11), (3.7)]{DavidLeger} only due to the additional truncation $\min\{\cdot,\frac{3}{2}\}$, requiring a single application of the chain rule for weak derivatives. Nevertheless, we sketch the full argument to make our paper self-contained.

To show \eqref{BeginMonotonicity} with the representation \eqref{DefinitionD1}, it is sufficient to observe that
\begin{align}
\label{FormulaDerivativeF}
&r \frac{d}{dr} F(r,x_0)
\\&
\nonumber
\geq 
\int_{\partial B_r(x_0)\setminus J_u} |\partial_\tau u|^2 + |\partial_\nu u|^2 \,d\mathcal{H}^1
+\frac{1}{2}\sum_{x\in J_u \cap \partial B_r(x_0)} \frac{1}{|\nu(x) \cdot \vec{t}_{J_u}(x)|} - F(r,x_0)
\end{align}
holds in the sense of distributions. Indeed, applying the chain rule for weak derivatives to this formula, we arrive at \eqref{BeginMonotonicity}-\eqref{DefinitionD1}.

As in \cite{DavidLeger}, the representation \eqref{DefinitionD2} follows directly from \eqref{DefinitionD1} using the David--L\'eger--Maddalena--Solimini relation \eqref{DavidLegerEtAl}. It therefore only remains to show \eqref{FormulaDerivativeF}.

To prove \eqref{FormulaDerivativeF}, we first write $E_{Dir}(r,x_0) := \frac{1}{r} \int_{B_r\setminus J_u} |\nabla u|^2 \,dx$. From Fubini's theorem and $\nabla u\in L^2(\Omega)$, we infer $E_{Dir}(r,x_0)=\frac{1}{r} \int_0^r \int_{\partial B_\rho \setminus J_u} |\nabla u|^2 \,d\mathcal{H}^1 \,d\rho$ and thus
\begin{align*}
\frac{d}{dr} E_{Dir}(r,x_0)
=\frac{1}{r}\int_{\partial B_r(x_0)\setminus J_u} |\partial_\tau u|^2 + |\partial_\nu u|^2 \,d\mathcal{H}^1 -\frac{1}{r} E_{Dir}(r).
\end{align*}
To establish \eqref{FormulaDerivativeF}, it hence suffices to show
\begin{align*}
\frac{d}{dr} \mathcal{H}^1(B_r\cap J_u)
\geq \sum_{x\in J_u \cap \partial B_r(x_0)} \frac{1}{|\nu(x) \cdot \vec{t}_{J_u}(x)|}
\end{align*}
in the sense of distributions. To see this, note that outside of a set of negligible $\mathcal{H}^1$ measure, the jump set $J_u$ is locally given by a $C^1$ curve \cite{AmbrosioFuscoPallaraPartialRegularity}. By Sard's theorem, this shows that for a.\,e.\ $r\in (0,\dist(x_0,\partial\Omega))$ the set $J_u\cap \partial B_r(x_0)$ is given by the intersection of a finite number of $C^1$ curves with $\partial B_r(x_0)$, with all of these intersections occurring transversally (i.\,e., the scalar product $\nu(x)\cdot \vec{t}_{J_u}(x)$ being nonzero). For a finite collection of $C^1$ curves, however, the estimate becomes an equality and is shown by a straightforward computation. Alternatively, to avoid the use of partial regularity of minimizers, the previous formula could also be obtained via the coarea formula for rectifiable sets \cite[Theorem 2.93]{AmbrosioFuscoPallara}.
\end{proof}

The main step for the derivation of the monotonicity formula is the derivation of the following estimates, which we will perform in the following subsections.
\begin{lemma}
\label{LemmaEstimates}
With the assumptions and the notation of Lemma~\ref{LemmaBeginMonotonicity}, we have
\begin{subequations}
\begin{align}
\label{EstimateZeroDiscontinuityPoints}
D(r,x_0)&\geq \chi_{F(r,x_0)<\tfrac{3}{2}} \frac{1}{2} \int_{\partial B_r(x_0)\setminus J_u} |\nabla u|^2 \,d\mathcal{H}^1
&\text{if }\#\{J_u\cap \partial B_r(x_0)\}=0,
\\
D(r,x_0)&\geq \big(\tfrac{3}{2}-F(r,x_0)\big)_+
\label{EstimateOneDiscontinuityPoint}
&\text{if }\#\{J_u\cap \partial B_r(x_0)\}=1,
\\
\label{EstimateTwoDiscontinuityPoints}
D(r,x_0)&\geq \chi_{F(r,x_0)<\tfrac{3}{2}} c \int_{\partial B_r(x_0)\setminus J_u} |\nabla u|^2 \,d\mathcal{H}^1
&\text{if }\#\{J_u\cap \partial B_r(x_0)\}=2,
\\
\label{EstimateThreeDiscontinuityPoints}
D(r,x_0)&\geq \big(\tfrac{3}{2}-F(r,x_0)\big)_+
&\text{if }\#\{J_u\cap \partial B_r(x_0)\}\geq 3.
\end{align}
\end{subequations}
\end{lemma}
\begin{proof}[Proof of Theorem~\ref{TheoremMonotonicityFormula}]
Our monotonicity formula follows immediately from Lemma~\ref{LemmaBeginMonotonicity} and Lemma~\ref{LemmaEstimates}, taking into consideration that our desired monotonicity follows directly from estimate \eqref{DefinitionD1} in case $\int_{\partial B_r(x_0)\setminus J_u} |\nabla u|^2 \,d\mathcal{H}^1\geq 2$.
\end{proof}

\subsection{The case of a single discontinuity point on $\partial B_r(x_0)$}

The situation of a single discontinuity point $J_u$ on the boundary $\partial B_r(x_0)$ is the main situation in which David and L\'eger \cite{DavidLeger} were missing a suitable argument for monotonicity. With our lower bound from Proposition~\ref{PropositionLowerBoundRadialDirichletEnergyCrack}, the monotonicity argument becomes trivial for the truncated entropy in this case:
\begin{proof}[Proof of \eqref{EstimateOneDiscontinuityPoint}]
By Proposition~\ref{PropositionLowerBoundRadialDirichletEnergyCrack}, we have
\begin{align*}
\int_{\partial B_r(x_0)\setminus J_u} |\nabla u|^2 \,d\mathcal{H}^1 \geq 1.
\end{align*}
Inserting this into formula \eqref{DefinitionD1}, we deduce \eqref{EstimateOneDiscontinuityPoint}.
\end{proof}

\subsection{The case of zero discontinuity points on $\partial B_r(x_0)$}

We handle this case just like in \cite{DavidLeger}, using the elementary competitor construction of Lemma~\ref{LemmaCompetitorZeroIntersection}.
\begin{proof}[Proof of \eqref{EstimateZeroDiscontinuityPoints}]
By Lemma~\ref{LemmaCompetitorZeroIntersection} we have
\begin{align*}
E(r,x_0)\leq \int_{\partial B_r(x_0)\setminus J_u} |\partial_\tau u|^2 \,d\mathcal{H}^1.
\end{align*}
Using formula \eqref{DefinitionD2}, we obtain
\begin{align*}
D(r,x_0)\geq \chi_{F(r,x_0)<\frac{3}{2}} \frac{1}{2} \int_{\partial B_r(x_0)\setminus J_u} |\nabla u|^2 \,d\mathcal{H}^1.
\end{align*}
\end{proof}

\subsection{The case of three or more discontinuity points on $\partial B_r(x_0)$}

An monotonicity estimate slightly different from \eqref{EstimateThreeDiscontinuityPoints} has already been established in David-L\'eger \cite{DavidLeger}. However, in our situation of the ``truncated'' entropy, a much shorter argument is available.
\begin{proof}[Proof of \eqref{EstimateThreeDiscontinuityPoints}]
The estimate is immediate from \eqref{DefinitionD1}, as we have
\begin{align*}
\frac{1}{2}\sum_{x\in J_u \cap \partial B_r(x_0)} \frac{1}{|\nu(x) \cdot \vec{t}_{J_u}(x)|}
\geq \frac{3}{2}.
\end{align*}
\end{proof}

\subsection{The case of two discontinuity points on $\partial B_r(x_0)$}

In the situation of two discontinuity points $J_u$ on the boundary $\partial B_r(x_0)$, David and L\'eger \cite{DavidLeger} only deduced monotonicity for certain angular configurations of the discontinuity points. We have to provide a new idea in the case that the shorter of the two circular arcs enclosed by the discontinuity points has length $\leq r (\frac{1}{2}\pi+0.00001)$.
\begin{proof}[Proof of \eqref{EstimateTwoDiscontinuityPoints}]
In the case that the longer of the two circular arcs enclosed by the discontinuity points $J_u\cap \partial B_r(x_0)$ has length of at most $\omega r \leq r(\frac{3}{2}\pi - 0.00001)$, we proceed as in \cite{DavidLeger} using the representation \eqref{DefinitionD2} and a competitor construction to bound the energy density $E(r,x_0)$: To construct a competitor $v$, we simply connect the two discontinuity points on $\partial B_r(x_0)$ to the center $x_0$ by a straight line and allow our competitor $v$ to have a discontinuity along these lines (see Figure~\ref{FigureTwoPoints}). In other words, we may construct $v$ separately in each of the two sectors enclosed by these lines and $\partial B_r(x_0)$. By Lemma~\ref{CompetitorEstimateSector}, this yields
\begin{align*}
E(r,x_0)
&\leq \frac{1}{r} \int_{B_r(x_0)\setminus J_v} |\nabla v|^2 \,dx + \frac{1}{r} \mathcal{H}^1(J_v \cap B_r(x_0))
\\
&\leq \frac{\frac{3}{2}\pi-0.00001}{\pi}
\int_{\partial B_r(x_0)\setminus J_u} |\partial_\tau u|^2 \,d\mathcal{H}^1 + 2
\\&
\leq \bigg(\frac{3}{2}-c\bigg)
\int_{\partial B_r(x_0)\setminus J_u} |\partial_\tau u|^2 \,d\mathcal{H}^1 + \#\{J_u\cap \partial B_r(x_0)\}.
\end{align*}
Thus, we may conclude using \eqref{DefinitionD2}.

In the remaining case that the longer of the two circular arcs enclosed by the points $J_u\cap \partial B_r(x_0)$ on $\partial B_r(x_0)$ has length $\omega r \geq (\tfrac{3}{2}\pi-0.00001)r$, our new idea is to exploit the consequences of Proposition~\ref{PropositionLowerBoundRadialDirichletEnergyCrack} in the form of Lemma~\ref{LemmaLowerBoundTwoPoints} below.
Note that in case that
\begin{align*}
\sum_{x\in J_u\cap \partial B_r(x_0)} \frac{1}{2|\nu(x)\cdot \vec{t}_{J_u}(x)|}
+\int_{\partial B_r(x_0)} |\partial_\tau u|^2 + |\partial_\nu u|^2 \,d\mathcal{H}^1
\geq \frac{3}{2} +0.005
\end{align*}
holds, by \eqref{DefinitionD1} we clearly have
\begin{align*}
D(r,x_0)\geq \chi_{F(r,x_0)<\frac{3}{2}} (1.505-1.5) \geq c \chi_{F(r,x_0)<\frac{3}{2}}.
\end{align*}
Thus, by Lemma~\ref{LemmaLowerBoundTwoPoints} the only case which we still need to consider is
\begin{subequations}
\begin{align}
\label{TwoPointOmega}
&~~~~~~~~~~~~~~~~~~~~\omega\leq \bigg(\frac{3}{2}+\frac{8}{180}\bigg)\pi
\\&~~~~~~~~~~~~~~~~~~~~~~~~~~~~~
\text{ and }\nonumber
\\&
\label{TwoPointCos}
~~~~~~~~~\sum_{x\in J_u\cap \partial B_r(x_0)} \frac{1}{2|\nu(x)\cdot \vec{t}_{J_u}(x)|}\geq 1.26
\\&~~~~~~~~~~~~~~~~~~~~~~~~~~~~~
\text{ and }\nonumber
\\&
\label{TwoPointCosCos}
\sum_{x\in J_u\cap \partial B_r(x_0)} \bigg(\frac{1}{2|\nu(x)\cdot \vec{t}_{J_u}(x)|}+\frac{1}{2}|\nu(x)\cdot \vec{t}_{J_u}(x)| \bigg) \geq 2+0.055.
\end{align}
\end{subequations}
Furthermore, we may restrict ourselves to the case
\begin{align}
\label{EstimateTangential}
\int_{\partial B_r(x_0)} |\partial_\tau u|^2 + |\partial_\nu u|^2 \,d\mathcal{H}^1
\leq 0.25,
\end{align}
as otherwise we again obtain the desired estimate on $D(r,x_0)$ by \eqref{TwoPointCos} and \eqref{DefinitionD1}. Inserting the estimate \eqref{TwoPointCosCos} into \eqref{DefinitionD2}, we deduce
\begin{align*}
D(r,x_0)
&\geq \chi_{F(r,x_0)<\tfrac{3}{2}} \bigg[2+0.055+\frac{3}{2} \int_{\partial B_r(x_0)} |\partial_\tau u|^2 \,d\mathcal{H}^1 - E(r,x_0)\bigg]
\\&
\stackrel{\eqref{EstimateTangential}}{\geq}
\chi_{F(r,x_0)<\tfrac{3}{2}}
\bigg[
0.005 + 2 + 1.7 \int_{\partial B_r(x_0)} |\partial_\tau u|^2 \,d\mathcal{H}^1 - E(r,x_0)\bigg]
\end{align*}
Due to \eqref{TwoPointOmega} and Lemma~\ref{LemmaCompetitorEstimateSector}, this enables us to conclude.
\end{proof}

\begin{lemma}
\label{LemmaLowerBoundTwoPoints}
Let $\Omega\subset\mathbb{R}^2$ be a planar domain and let $u$ be a minimizer to the Mumford-Shah functional. Let $x_0\in \Omega$.

For a.\,e.\ $0<r<\dist(x_0,\partial\Omega)$ with $\#\{J_u\cap \partial B_r(x_0)\}=2$ and $0<\varphi \leq \frac{1}{2}\pi+0.00001$, we have
\begin{align*}
\sum_{x\in J_u\cap \partial B_r(x_0)} \frac{1}{2|\nu(x)\cdot \vec{t}_{J_u}(x)|}
+\int_{\partial B_r(x_0)} |\partial_\tau u|^2 + |\partial_\nu u|^2 \,d\mathcal{H}^1
\geq \sqrt{2}-0.005,
\end{align*}
where $r\varphi$ denotes the length of the shorter of the two circular arcs enclosed by the points $J_u\cap \partial B_r(x_0)$ on $\partial B_r(x_0)$.

If $\varphi<\frac{1}{2}\pi-\frac{8}{180}\pi$ or
\begin{align*}
\sum_{x\in J_u\cap \partial B_r(x_0)} \frac{1}{2|\nu(x)\cdot \vec{t}_{J_u}(x)|} < 1.26
\end{align*}
or
\begin{align*}
\sum_{x\in J_u\cap \partial B_r(x_0)} \Big(\frac{1}{2|\nu(x)\cdot \vec{t}_{J_u}(x)|} +\frac{|\nu(x)\cdot \vec{t}_{J_u}(x)|}{2} \Big) < 2+0.055
\end{align*}
holds, we even have
\begin{align*}
\sum_{x\in J_u\cap \partial B_r(x_0)} \frac{1}{2|\nu(x)\cdot \vec{t}_{J_u}(x)|}
+\int_{\partial B_r(x_0)} |\partial_\tau u|^2 + |\partial_\nu u|^2 \,d\mathcal{H}^1
\geq \frac{3}{2}+0.005.
\end{align*}
\end{lemma}
\begin{proof}
We denote the two points in $J_u\cap \partial B_r(x_0)$ by $x_1$ and $x_2$; without loss of generality, we may assume that $x_1$ lies counterclockwise from $x_2$. We then define $\alpha_i\in (-\frac{\pi}{2},\frac{\pi}{2})$ by requiring $\vec{t}_{J_u}(x_i):=\cos(\alpha_i) \nu(x_i)+\sin(\alpha_i) \nu^\perp(x_i)$, where the sign convention of $^\perp$ is such that $e_1^\perp=e_2$.
Observe that the angle between the vectors $\vec{t}_{J_u}(x_1)$ and $\vec{t}_{J_u}(x_2)$ is given by $\varphi+\alpha_1-\alpha_2$; thus, the length of the vector $\vec{t}_{J_u}(x_1)+\vec{t}_{J_u}(x_2)$ is given as $\sqrt{2+2\cos(\varphi+\alpha_1-\alpha_2)}$.
Using Proposition~\ref{PropositionLowerBoundRadialDirichletEnergyCrack} and optimizing in $\vec{q}$, we deduce
\begin{align}
\nonumber
\int_{\partial B_r(x_0)} |\partial_\tau u|^2 + |\partial_\nu u|^2 \,d\mathcal{H}^1
&\geq
\sqrt{2+2\cos(\varphi+\alpha_1-\alpha_2)}
\\&
\label{LowerBoundGradTwoPoints}
\geq
\sqrt{2+2\cos(\tilde \varphi+\alpha_1-\alpha_2)}-\sqrt{0.00002}
\end{align}
with $\tilde \varphi := \min\{\varphi,\tfrac{1}{2}\pi\}$. Here, in the last step we have used elementary bounds, the Lipschitz continuity of $\cos$ with Lipschitz constant $1$, as well as the assumption $\varphi\leq \tfrac{\pi}{2}+0.00001$.

We then consider the minimization problem for
\begin{align}
\label{Definitionf}
f(\alpha_1,\alpha_2)
:=\frac{1}{2\cos \alpha_1} + \frac{1}{2\cos \alpha_2} + \sqrt{2+2\cos(\tilde \varphi+\alpha_1-\alpha_2)}
\end{align}
for $-\pi/2<\alpha_1,\alpha_2<\pi/2$. Using the fact that $0<\tilde \varphi \leq \pi/2$, the monotonicity of $\cos(s)$ for $0<s<\pi$, and the symmetries of $\cos(s)$, by considering all combinations of signs $\pm \alpha_i$ we observe that the infimum is not changed upon restricting to $\alpha_1\geq 0$ and $\alpha_2\leq 0$. Set $\alpha:=(\alpha_1-\alpha_2)/2$ and $s:=(\alpha_1+\alpha_2)/2$; we then have $0<\alpha<\pi/2$ and $-\alpha<s<\alpha$ and the function may be rewritten as
\begin{align*}
f(\alpha_1,\alpha_2)
=\frac{1}{2\cos(\alpha+s)}
+\frac{1}{2\cos(\alpha-s)}
+\sqrt{2+2\cos(\tilde \varphi+2\alpha)}.
\end{align*}
As for fixed $\alpha$ the function on the right-hand side is convex in $s$ (for $-\alpha<s<\alpha$) and symmetric around $s=0$, the minimum in $s$ is obtained for $s=0$. This shows
\begin{align*}
f(\alpha_1,\alpha_2)
&\geq
\frac{1}{2\cos \alpha}
+\frac{1}{2\cos \alpha}
+\sqrt{2+2\cos(\tilde \varphi+2\alpha)}
\\&
=\frac{1}{\cos \alpha} + 2\Big|\cos\Big(\frac{\tilde \varphi}{2}+\alpha\Big)\Big|
\\&
\geq 
\frac{1}{\cos \alpha} + 2\Big(\cos\Big(\frac{\tilde \varphi}{2}+\alpha\Big)\Big)_+
\end{align*}
Since $0<\tilde \varphi \leq \pi/2$ and $0<\alpha<\pi/2$, we see that the second term is nonincreasing in $\tilde \varphi$. We thus get
\begin{align*}
f(\alpha_1,\alpha_2)
\geq 
\frac{1}{\cos \alpha} + 2\Big(\cos\Big(\frac{82}{90} \cdot \frac{\pi}{4}+\alpha\Big)\Big)_+
\end{align*}
for $0<\tilde \varphi< \frac{82}{90} \cdot \frac{\pi}{2}$ and
\begin{align}
\label{LowerBoundf}
f(\alpha_1,\alpha_2)
\geq 
\frac{1}{\cos \alpha} + 2\Big(\cos\Big(\frac{\pi}{4}+\alpha\Big)\Big)_+
\end{align}
for all $0<\tilde \varphi\leq \frac{\pi}{2}$.
The minimum of these right-hand sides is readily seen to be attained at $\alpha=\frac{98}{90}\cdot \tfrac{\pi}{4}$ respectively at $\alpha=\tfrac{\pi}{4}$ (for smaller values of $\alpha$ the derivative of the right-hand side is $\sin\alpha/\cos^2\alpha-2\sin(82/90\cdot \pi/4+\alpha)<0$ respectively $\sin\alpha/\cos^2\alpha-2\sin(\pi/4+\alpha)<0$; for larger values of $\alpha$, the derivative is positive). We therefore get
\begin{align*}
f(\alpha_1,\alpha_2)\geq \frac{1}{\cos \big(\frac{98}{90}\cdot\tfrac{\pi}{4}\big)}\geq 1.52
\end{align*}
for $0<\tilde \varphi<\frac{82}{90}\cdot \tfrac{\pi}{2}$ and
\begin{align*}
f(\alpha_1,\alpha_2)\geq \sqrt{2}
\end{align*}
in the general case $0<\tilde \varphi\leq \tfrac{\pi}{2}$. Furthermore, by \eqref{LowerBoundf} and the considerations just below \eqref{LowerBoundf} we see that $f(\alpha_1,\alpha_2)<1.51$ entails $(\alpha_1-\alpha_2)/2>0.42 \cdot \pi/2$. This in turn implies $\frac{1}{2\cos \alpha_1}+ \frac{1}{2\cos \alpha_2} \geq 1.26$ as well as $\sum_{i=1}^2\big(\frac{1}{2\cos \alpha_i}+\frac{1}{2}\cos \alpha_i\big)\geq 2.055$ due to convexity and monotonicity of the functions $\tilde g(\tilde \alpha)=\frac{1}{\cos \tilde \alpha}$ and $g(\tilde \alpha):=\frac{1}{\cos \tilde \alpha}+\cos \tilde \alpha$.

Our lemma therefore follows from \eqref{LowerBoundGradTwoPoints}, \eqref{Definitionf}, and the lower bounds on $f$ derived above.
\end{proof}

\begin{figure}
\begin{tikzpicture}[scale=0.9]
\draw (0,0) circle (2.0);
\draw[color=Blue,very thick] (0,0) -- (1.2,-1.6);
\draw[color=Blue,very thick] (0,0) -- (1.2,1.6);
\draw[fill=black] (1.2,1.6) circle (1.2pt);
\draw[fill=black] (1.2,-1.6) circle (1.2pt);
\end{tikzpicture}
\caption{
\label{FigureTwoPoints}
Construction of a competitor for two discontinuity points on the boundary: If the larger arc enclosed by the two points in $\{J_u \cap \partial B_r\}$ has length $\leq (\frac{3}{2}+0.00001)\pi r$, we allow our competitor $v$ to have a discontinuity along the lines connecting the two points in $\{J_u \cap \partial B_r\}$ to the center of the disk.}
\end{figure}
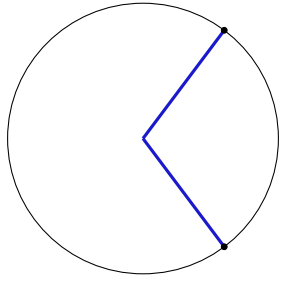

\subsection{Construction of competitors on circular sectors}

The following estimate on harmonic functions in circular sectors is well-known and has been used e.\,g.\ by David and L\'eger \cite{DavidLeger} to construct competitors for minimizers of the Mumford-Shah functional.
\begin{lemma}
\label{LemmaCompetitorEstimateSector}
Let $r>0$. Consider a circular sector $S\subset B_r(x_0)$ with opening angle $0<\theta<2\pi$. Let $u\in H^1(\partial B_r(x_0)\cap \overline S)$ be any function defined on the outer boundary of the sector. Then there exists an extension $v\in H^1(S)$ with $v|_{\partial B_r(x_0)\cap \overline S} = u|_{\partial B_r(x_0)\cap \overline S}$ and
\begin{align}
\label{CompetitorEstimateSector}
\frac{1}{r} \int_S |\nabla v|^2 \,dx
\leq \frac{\theta}{\pi} \int_{\partial B_r(x_0)\cap \overline S} |\partial_\tau u|^2 \,d\mathcal{H}^1.
\end{align}
\end{lemma}
\begin{proof}
Assume without loss of generality that the circular sector $S$ is given in radial coordinates as $\{(\rho,\varphi):0<\rho<r,0<\varphi<\theta\}$.
We expand $u|_{\partial B_r(x_0)\cap \overline S}$ in terms of its Fourier modes as
\begin{align*}
u|_{\partial B_r(x_0)\cap \overline S}(\rho,\varphi)
=\sum_{k=0}^\infty a_k \cos\bigg(\frac{k\pi}{\theta}\varphi\bigg).
\end{align*}
This enables us to define
\begin{align*}
v(\rho,\varphi) := \sum_{k=0}^\infty a_k \bigg(\frac{\rho}{r}\bigg)^{k\pi/\theta} \cos\bigg(\frac{k\pi}{\theta}\varphi\bigg).
\end{align*}
We compute
\begin{align*}
\int_{\partial B_r(x_0)\cap \overline S} |\partial_\tau u|^2 \,d\mathcal{H}^1
=r^{-1} \frac{\pi^2}{2\theta} \sum_{k=1}^\infty k^2 |a_k|^2
\end{align*}
and
\begin{align*}
\int_{S} |\nabla v|^2 \,dx
=\int_0^r \rho \theta \times \frac{\rho^{2k\pi/\theta-2}}{r^{2k\pi/\theta}} \sum_{k=1}^\infty \frac{k^2 \pi^2}{\theta^2} |a_k|^2 \,d\rho
=\frac{\pi}{2} \sum_{k=1}^\infty k |a_k|^2,
\end{align*}
which proves \eqref{CompetitorEstimateSector}.
\end{proof}

\section*{Acknowledgements}
\nopagebreak
\begin{center}
\includegraphics[scale=0.1]{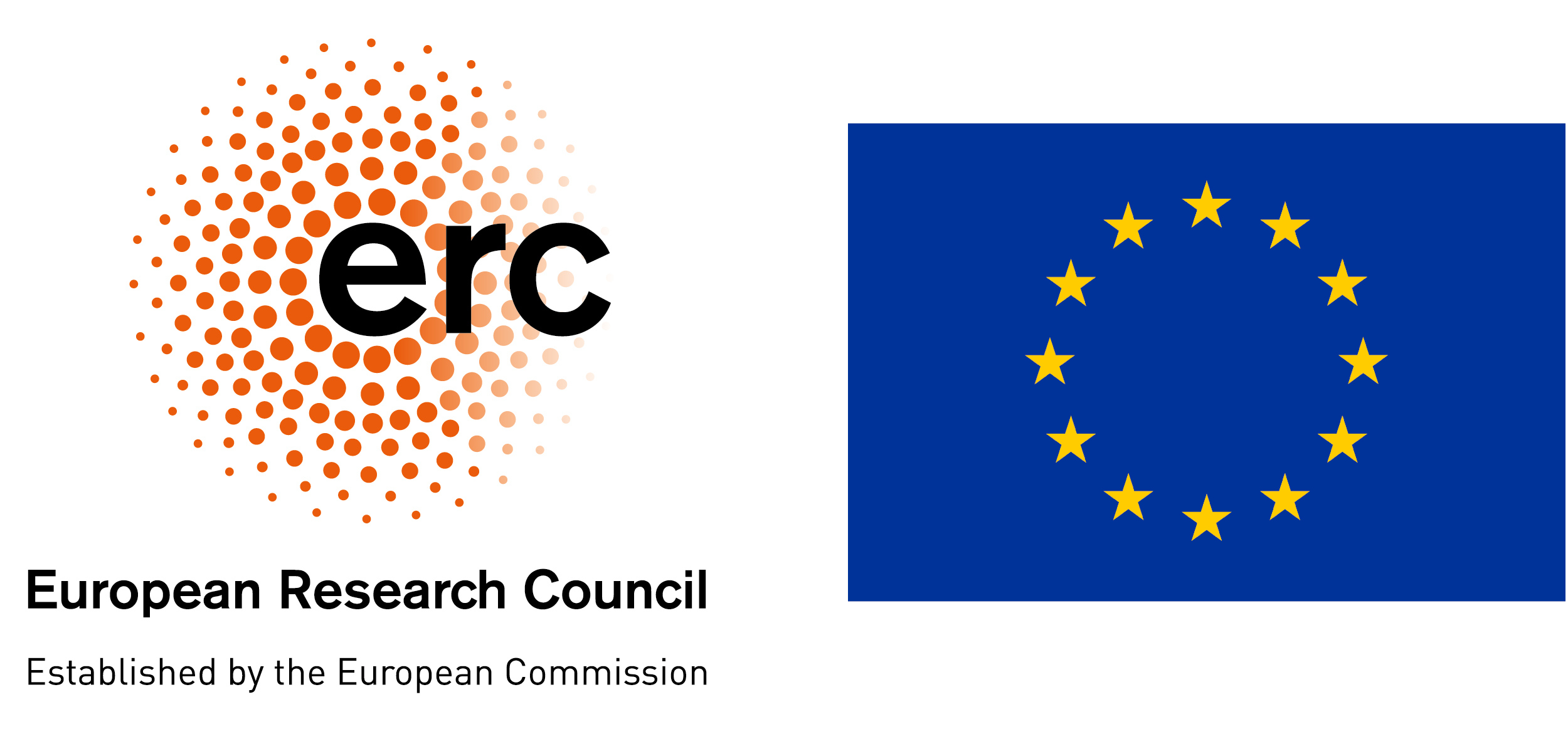} 
\end{center}
This project has received funding from the European Research Council (ERC) under the European Union's Horizon 2020 research and innovation programme (grant agreement No 948819).

\bibliographystyle{abbrv}
\bibliography{mumford_shah}

\begin{thebibliography}{10}

\bibitem{AlbertiDalMasoCalibration}
G.~Alberti, G.~Bouchitt\'{e}, and G.~Dal~Maso.
\newblock The calibration method for the {M}umford-{S}hah functional and
  free-discontinuity problems.
\newblock {\em Calc. Var. Partial Differential Equations}, 16(3):299--333,
  2003.

\bibitem{Ambrosio}
L.~Ambrosio.
\newblock Variational problems in {SBV} and image segmentation.
\newblock {\em Acta Appl. Math.}, 17(1):1--40, 1989.

\bibitem{AmbrosioFuscoHutchinson}
L.~Ambrosio, N.~Fusco, and J.~Hutchinson.
\newblock Higher integrability of the gradient and dimension of the singular
  set for minimisers of the {Mumford-Shah} functional.
\newblock {\em Cal. Var. Partial Diff. Eq.}, 16:187--215, 2003.

\bibitem{AmbrosioFuscoPallaraPartialRegularity}
L.~Ambrosio, N.~Fusco, and D.~Pallara.
\newblock Partial regularity of free discontinuity sets {II}.
\newblock {\em Ann. Sc. Norm. Sup. Pisa}, 24:39--62, 1997.

\bibitem{AmbrosioFuscoPallara}
L.~Ambrosio, N.~Fusco, and D.~Pallara.
\newblock {\em Functions of Bounded Variation and Free Discontinuity Problems
  (Oxford Mathematical Monographs)}.
\newblock Oxford University Press, 2000.

\bibitem{AmbrosioPallaraPartialRegularity}
L.~Ambrosio and D.~Pallara.
\newblock Partial regularity of free discontinuity sets {I}.
\newblock {\em Ann. Sc. Norm. Super. Pisa Cl. Sci.}, 24(1):1--38, 1997.

\bibitem{AnderssonMikayelyanEndpoint}
J.~Andersson and H.~Mikayelyan.
\newblock {Regularity up to the Crack-Tip for the Mumford-Shah problem}.
\newblock {\em Preprint}, 2015.
\newblock arXiv:1512.05094.

\bibitem{Bonnet}
A.~Bonnet.
\newblock On the regularity of edges in image segmentation.
\newblock In {\em Ann. Inst. H. Poincar{\'e} Anal. non lin{\'e}aire},
  volume~13, pages 485--528. Elsevier, 1996.

\bibitem{BonnetDavid}
A.~Bonnet and G.~David.
\newblock Cracktip is a global {Mumford-Shah} minimizer.
\newblock {\em Ast{\'e}risque}, 274, 2001.

\bibitem{BucurLuckhaus}
D.~Bucur and S.~Luckhaus.
\newblock Monotonicity formula and regularity for general free discontinuity
  problems.
\newblock {\em Arch. Ration. Mech. Anal.}, 211:489--511, 2014.

\bibitem{ChambolleContiIurlano}
A.~Chambolle, S.~Conti, and F.~Iurlano.
\newblock Approximation of functions with small jump sets and existence of
  strong minimizers of {G}riffith's energy.
\newblock {\em J. Math. Pures Appl.}, 128:119--139, 2019.

\bibitem{DalMasoFrancfortToader}
G.~Dal~Maso, G.~A. Francfort, and R.~Toader.
\newblock Quasistatic crack growth in nonlinear elasticity.
\newblock {\em Arch. Ration. Mech. Anal.}, 176(2):165--225, 2005.

\bibitem{DalMasoMorelSolimini}
G.~Dal~Maso, J.-M. Morel, and S.~Solimini.
\newblock A variational method in image segmentation: existence and
  approximation results.
\newblock {\em Acta Math.}, 168(1):89--151, 1992.

\bibitem{DavidLeger}
G.~David and J.-C. L\'{e}ger.
\newblock Monotonicity and separation for the {M}umford-{S}hah problem.
\newblock {\em Ann. Inst. H. Poincar\'{e} Anal. Non Lin\'{e}aire},
  19(5):631--682, 2002.

\bibitem{DeGiorgiConjectures}
E.~{De Giorgi}.
\newblock Free discontinuity problems in calculus of variations.
\newblock {\em Frontiers in pure and applied Mathematics, a collection of
  papers dedicated to JL Lions on the occasion of his 60th birthday, R. Dautray
  ed., North Holland}, 1991.

\bibitem{DeGiorgiCarrieroLeaci}
E.~{De Giorgi}, M.~Carriero, and A.~Leaci.
\newblock Existence theorem for a minimum problem with free discontinuity set.
\newblock {\em Ennio De Giorgi}, page 654, 1989.

\bibitem{DeLellisFocardi}
C.~{De Lellis} and M.~Focardi.
\newblock Density lower bound estimates for local minimizers of the 2d
  {Mumford-Shah} energy.
\newblock {\em Manuscripta Math.}, 142:215--232, 2013.

\bibitem{DeLellisFocardiHigherIntegrability}
C.~{De~Lellis} and M.~Focardi.
\newblock Higher integrability of the gradient for minimizers of the 2d
  {Mumford--Shah} energy.
\newblock {\em J. Math. Pures Appl.}, 100(3):391--409, 2013.

\bibitem{DeLellisFocardiGhinassiEndpoint}
C.~{De~Lellis}, M.~Focardi, and S.~Ghinassi.
\newblock {Endpoint regularity for 2d Mumford-Shah minimizers: On a theorem of
  Andersson and Mikayelyan}.
\newblock {\em to appear in J. Math. Pures Appl.}, 2022.
\newblock arXiv:2010.04888.

\bibitem{DePhilippisFigalli}
G.~{De Philippis} and A.~Figalli.
\newblock Higher integrability for minimizers of the {Mumford-Shah} functional.
\newblock {\em Arch. Ration. Mech. Anal.}, 213:491--502, 2013.

\bibitem{FrancfortMarigo}
G.~A. Francfort and J.-J. Marigo.
\newblock Revisiting brittle fracture as an energy minimization problem.
\newblock {\em Journal of the Mechanics and Physics of Solids},
  46(8):1319--1342, 1998.

\bibitem{MaddalenaSolimini}
F.~Maddalena and S.~Solimini.
\newblock Blow-up techniques and regularity near the boundary for free
  discontinuity problems.
\newblock {\em Adv. Nonlinear Studies}, 1(2), 2001.

\bibitem{MumfordShah}
D.~B. Mumford and J.~Shah.
\newblock Optimal approximations by piecewise smooth functions and associated
  variational problems.
\newblock {\em Comm. Pure Appl. Math.}, 1989.

\end{thebibliography}

\end{document}